\documentclass[smallextended]{svjour3}       

\usepackage{amsmath,amssymb,color,graphicx,epstopdf}

\allowdisplaybreaks
\smartqed  

\newcounter{num}

\newenvironment{enumer}
  {\begin{list}
    {\textbf{{\sl (\roman{num})}}}
    {\usecounter{num}
     \setlength{\topsep}{0.5ex}
     \setlength{\partopsep}{4ex}
     \setlength{\leftmargin}{5.5ex}
     \setlength{\labelwidth}{6ex}
     \setlength{\labelsep}{1ex}
     \setlength{\itemsep}{0ex}
     \setlength{\parsep}{0ex}
     \setlength{\itemindent}{0ex}
     \setlength{\listparindent}{0ex}
    }
   }
  {\end{list}}

\newcommand{\R}{\mathbb{R}}
\newcommand{\rmd}{\mathrm{d}}

\newcommand{\fa}{\quad \mbox{for all }\,}

\begin{document}

\title{Transit times and mean ages for nonautonomous and autonomous compartmental systems}
\titlerunning{Transit times and mean ages for compartmental systems}

\author{Martin Rasmussen \and Alan Hastings \and Matthew J.~Smith \and Folashade B.~Agusto \and Benito M.~Chen-Charpentier \and Forrest M.~Hoffman \and Jiang Jiang \and Katherine E.O.~Todd-Brown \and Ying Wang \and Ying-Ping Wang \and Yiqi Luo}

\date{\today}

\authorrunning{Martin Rasmussen et al.}

\institute{
Martin Rasmussen, Department of Mathematics, Imperial College London, UK \and
Alan Hastings, Department of Environmental Science and Policy, University of California, Davis, USA \and
Matthew J.~Smith, Computational Science Laboratory, Microsoft Research, Cambridge, UK \and
Folashade B.~Agusto, Department of Mathematics and Statistics, Austin Peay State University, USA \and
Benito M.~Chen-Charpentier, Department of Mathematics, University of Texas, Arlington, Texas, USA \and
Forrest M.~Hoffman, Oak Ridge National Laboratory, Computational Earth Sciences Group, Tennessee, USA \and
Jiang Jiang, Department of Microbiology and Plant Biology, University of Oklahoma, Norman, Oklahoma, USA\and
Katherine E.O.~Todd-Brown, Earth and Biological Sciences Directorate, Pacific Northwest National Laboratory, Richland, State of Washington, USA,
and Department of Microbiology and Plant Biology, University of Oklahoma, Norman, Oklahoma, USA \and
Ying Wang, Department of Mathematics, University of Oklahoma, Norman, Oklahoma, USA \and
Ying-Ping Wang, CSIRO Ocean and Atmosphere Flagship, Victoria, Australia \and
Yiqi Luo, Department of Microbiology and Plant Biology, University of Oklahoma, Norman, Oklahoma, USA}

\maketitle

\begin{abstract}
   We develop a theory for transit times and mean ages for nonautonomous compartmental systems. Using the McKendrick--von F{\"o}rster equation, we show that the mean ages of mass in a compartmental system satisfy a linear nonautonomous ordinary differential equation that is exponentially stable. We then define a nonautonomous version of transit time as the mean age of mass leaving the compartmental system at a particular time and show that our nonautonomous theory generalises the autonomous case. We apply these results to study a nine-dimensional nonautonomous compartmental system modeling the terrestrial carbon cycle, which is a modification of the Carnegie--Ames--Stanford approach (CASA) model, and we demonstrate that the nonautonomous versions of transit time and mean age differ significantly from the autonomous quantities when calculated for that model.
\end{abstract}

\keywords{Carbon cycle, CASA model, Compartmental system, Exponential stability, Linear system, McKendrick--von F{\"o}rster equation, Mean age, Nonautonomous dynamical system, Transit time}

\subclass{34A30, 34D05}

\section{Introduction}

Compartment models play an important role in the modeling of many biological systems ranging from pharmacokinetics to ecology \cite{Anderson_83_1,Godfrey_83_1,Jacquez_93_1}.  Key values in understanding the dynamics of these systems are the transit time: the mean time a particle spends in the compartmental system measured as the mean time from entry into the system to leaving the system \cite{Bolin_73_1,Eriksson_71_1}, and the mean age: the mean age of particles still in the system \cite{Bolin_73_1,Eriksson_71_1}. It is well known that these quantities need not be the same \cite{Bolin_73_1,Eriksson_71_1,Rothman_15_1}.

We are motivated by an interest in studying the dynamics of the terrestrial carbon cycle which is typically modeled as a number of discrete pools of carbon in plant biomass, litter and soil organic matter. Many of the best studied models of the dynamics of carbon are linear, which reflects the fact that changes in carbon pools are proportional to the pool size \cite{Bolker_98_1}.  Additionally, most analyses make the further assumption that all parameters describing the dynamics (and the input fluxes) are constant in time, leading to a model in the form of an autonomous linear differential equation.  In this autonomous case, it is possible to derive analytic formulae giving expressions for the transit time \cite{GarciaMeseguer_03_1,Manzoni_09_1}. These formulae for transit time are given in terms of (constant) transfer coefficients among compartments and analogous formulae are available for the mean age of particles in the system.

Many applications of models of terrestrial carbon relate to situations in which constant model parameters are replaced by time-dependent functions. Perhaps the most well-known examples are studies of how terrestrial carbon dynamics respond to climate change. In these, it is often assumed that the specific rates (per unit carbon) of carbon inputs and losses from the system change over time as a function of changes in climate, such as temperature. For example, increases in temperature are normally assumed to increase the rates of soil decomposition \cite{Lloyd_94_1,Orchard_83_1,Rothman_15_1}. As a consequence, the compartmental models of interest are nonautonomous, i.e.~they depend on time \cite{Luo_15_1,Luo_01_1,Xia_12_1}. Nonautonomous compartmental systems are special cases of linear nonautonomous differential equations \cite{Kloeden_11_2}, which, in contrast to the linear autonomous case, cannot be solved analytically in general. Yet, both the mean age of particles in the system and the transit time remain of great interest for these time-dependent systems, as both quantities can be potentially measured in the actual systems being modeled \cite{Rothman_15_1,Trumbore_00_1}.

In this paper, we develop a theory for transit times and mean ages of mass in nonautonomous compartmental systems. As noted in one of the first papers to study transit time \cite{Bolin_73_1}, there is obviously a close connection between age distribution and transit time in compartment models. We will build on this relationship to develop an approach for understanding the definition of transit time. We define a time-dependent version of transit time as the mean age of mass leaving the compartmental system. We use a time-dependent version of the McKendrick--von F{\"o}rster equation \cite{Brauer_12_1,McKendrick_26_1,Thieme_03_1}, the classic first-order partial differential equation describing age distributions, to prove that the mean age of mass satisfies an (inhomogeneous) linear nonautonomous differential equation. We show that under weak conditions, this equation is exponentially stable. Starting with demographic models highlights another important aspect of our approach. As is well known, solutions of demographic models depend on initial conditions, so quantities like the mean age and transit time also depend on initial conditions, but conventional definitions of these quantities ignore the influence of the initial conditions. For this reason, our nonautonomous approach also provides additional insight for autonomous compartmental systems that are not in equilibrium.

We apply the theory we have developed to numerically study transit times for a nine-dimensional compartmental system model of the carbon cycle, which is a modified version of the Carnegie--Ames--Stanford approach (CASA) model \cite{Buermann_07_1,Potter_93_1,Randerson_96_1}. We compare our nonautonomous quantities to the classical notion of transit time for autonomous systems, where we freeze the nonautonomous system in time to obtain an autonomous system, and we assume that we are in equilibrium. Our simulations illustrate the different and sometimes diverging trajectories of the autonomous and nonautonomous quantities over time. Our results demonstrate the necessity of our theory for the computation of transit times in nonautonomous compartmental systems and in autonomous compartmental systems that are not in equilibrium.

This paper is organized as follows. In Section~\ref{sec3}, we first review the theory of transit times for autonomous compartmental systems, and we provide a heuristic derivation of the transit time formula. We then define nonautonomous compartmental systems in Section~\ref{sec1}. In Section~\ref{sec2}, we prove that under the assumption that the compartmental system is lower block triangular, and the diagonal blocks a diagonally dominant, the nonautonomous compartmental system is exponentially stable. In Section~\ref{sec4}, we prove that the mean ages satisfy a linear nonautonomous differential equation, and we then use the stability criterion from Section~\ref{sec2} to prove exponential stability of the mean age equation. We define the concept of a transit time for nonautonomous compartmental systems in Section~\ref{sec5}. In Section~\ref{sec6}, we show that our nonautonomous theory is consistent with the autonomous case, in the sense that we get exactly the well-known autonomous transit time formula when applying the nonautonomous transit time to an autonomous system. Finally, in Section~\ref{sec7}, we apply the theory to compute transit times for a nonautonomous compartmental model of the carbon cycle, which is a simplified version of the Carnegie--Ames--Stanford approach (CASA) model.

\section{Transit times and mean ages for autonomous compartmental systems}\label{sec3}

An open (linear) \emph{autonomous compartmental system} with both inputs and outputs \cite{Anderson_83_1,Godfrey_83_1,Jacquez_93_1} and with $d$ pools is described by an inhomogeneous linear differential equation
\begin{equation}\label{eqn8}
  \dot x = B x + s\,,
\end{equation}
where $B \in \R^{d\times d}$ is an invertible matrix, $0\not=s\in [0,\infty)^d$, and the entries $\{b_{ij}\}_{i,j\in\{1,\dots,d\}}$ of the matrix $B$ satisfy
\begin{enumer}
  \item[$\bullet$] $b_{ii} < 0$ for all $i\in\{1,\dots,d\}$,
  \item[$\bullet$] $b_{ij} \ge 0$ for all $i\not=j\in\{1,\dots,d\}$,
  \item[$\bullet$] $\sum_{i=1}^d b_{ij} \le 0$ for all $j\in\{1,\dots,d\}$.
\end{enumer}

The $i$-th row of the matrix $B$ describes the dynamics of the mass in pool $i$: $b_{ij}$ is the rate at which mass moves from pool~$j$ to pool~$i$, and $b_{ii}$ is the rate at which mass leaves the pool $i$ which includes transfer to other pools and losses from the system. The flux at which mass enters from outside the system to pool~$i$ is given by $s_i$.

We assume that the homogeneous linear system $\dot x = Bx$ is exponentially stable, i.e.~all eigenvalues of $B$ have negative real parts (this is fulfilled e.g.~when the matrix $B$ is strictly diagonally dominant). This means that \eqref{eqn8} has the exponentially stable equilibrium $x^* = -B^{-1}s$.

The concept of transit time for compartmental systems describes the mean time a particle spends in the compartmental system before it is released. There is a huge amount of literature on this topic, see e.g.~\cite{Anderson_83_1,Bolin_73_1,Eriksson_71_1,GarciaMeseguer_03_1,Manzoni_09_1}, but to our knowledge, the following simple derivation of the transit time formula has not been written down before.

Define $r_i$ as the mean (remaining) transit time in the system for a particle that has entered pool $i$ either from outside the system or from another pool, and note that the transit time in pool $i$ for a particle that has entered pool $i$ either from outside the system or from another pool is given by $-\frac{1}{b_{ii}}$.
Let $p_{ij}$ be the probability that a particle that enters pool $i$ goes next to pool $j$, and note that
\begin{displaymath}
  p_{ij} = -\frac{b_{ji}}{b_{ii}}\,.
\end{displaymath}

Next, note that the transit times for particles entering any pool $i$ must satisfy the equation
\begin{displaymath}
  r_i = -\frac{1}{b_{ii}} + \sum_{j \neq i} p_{ij}r_j\,,
\end{displaymath}
reflecting the fact that a particle in pool $i$ spends the average time $-\frac{1}{b_{ii}}$ in pool $i$, before it either leaves the system or moves with the probability $p_{ij}$ to pool $j$, after which it spends the mean time $r_j$ before it leaves the system.
This reads as
\begin{displaymath}
  r = \begin{pmatrix}
     0 & p_{12} &  & p_{1d}  \\
     p_{21}  & 0 &&p_{2d}\\
     & & \ddots & \\
     p_{d1} & p_{d2} &&0
     \end{pmatrix}
     r
     -
     \begin{pmatrix}
       \frac{1}{b_{11}} \\
       \vdots \\
       \frac{1}{b_{dd}}
     \end{pmatrix}\,,
\end{displaymath}
and multiplying the $i$-th row of this equation with $-b_{ii}$ yields
\begin{displaymath}
  0 = B^T r + (1,\dots,1)^T\,.
\end{displaymath}
Hence $r^T= -(1,\dots,1) B^{-1}$.

Let $\beta_i$ be the fraction of particles that enter the system from outside directly into pool $i$, i.e.
\begin{displaymath}
  \beta_i = \frac{s_i}{\sum_{i=1}^d s_i} \fa i\in\{1,\dots,d\}\,,
\end{displaymath}
and let $\beta=(\beta_1,\dots,\beta_d)^T$. Then the transit time for the whole system is given by
\begin{equation}\label{restimeautonomous}
  R = -r^T\beta = -(1,\dots,1) B^{-1} \beta\,.
\end{equation}
Note that this transit time is equal to the turnover time $U= \frac{(1,\dots,1)(x_1^*, \dots, x_d^*)^T}{(1,\dots,1)(s_1, \dots, s_d)^T}$ (see \cite{Bolin_73_1}),
which follows directly from $x^* = -B^{-1}s$.

We will show later that if the linear compartmental system \eqref{eqn8} is in the equilibrium $x^* = - B^{-1}s$, then the mean age of the particles in the system is given by
\begin{equation}\label{meanagesautonomous}
  M = -(1,\dots,1) B^{-1} \eta \,,
\end{equation}
where $\eta = (\eta_1,\dots,\eta_d)^T$, defined by
\begin{displaymath}
  \eta_i = \frac{x_i^*}{\sum_{j=1}^d x_j^*} \fa i\in\{1,\dots,d\}\,,
\end{displaymath}
describes how mass is distributed when the system is in equilibrium. It is well-known that the mean age $M$ is unequal to the transit time $R$ \cite{Bolin_73_1,Rothman_15_1}, and we will demonstrate this now by means of two very simple compartmental systems.

\begin{example}[Transit times and mean ages]\label{exam5}
  Consider the two compartmental systems
  \begin{equation}\label{eqn9}
    \dot x = \begin{pmatrix}
               -1 & 2 \\
               0.5 & -2
             \end{pmatrix} x + \begin{pmatrix}
                                 1 \\
                                 0
                               \end{pmatrix}
  \end{equation}
  and
  \begin{equation}\label{eqn10}
    \dot x = \begin{pmatrix}
               -1 & 1 \\
               1 & -2
             \end{pmatrix} x + \begin{pmatrix}
                                 1 \\
                                 0
                               \end{pmatrix}\,.
  \end{equation}
  It is easy to see that the transit times $r_1$ and $r_2$ for the two pools satisfies $r_1< r_2$ for \eqref{eqn9} and $r_1 > r_2$ for \eqref{eqn10}. This follows either from using the above explicit formula for the vector $r$, or by considering the fact that particles can only leave from pool $1$ in \eqref{eqn9} and from pool $2$ in \eqref{eqn10}. Since the transit time is given in both cases by $r_1$, and the mean age is a convex combination of $r_1$ and $r_2$, the transit time will be smaller than the mean age in \eqref{eqn9}, in contrast to the situation in \eqref{eqn10}.
\end{example}

\section{Nonautonomous compartmental systems}\label{sec1}

In contrast to the autonomous case, both the coefficient matrix $B$ and the input vector $s$ of a nonautonomous compartmental system are allowed to depend on time.

\begin{definition}[Nonautonomous compartmental system]\label{def2}
  Let $I:=(\tau,\infty)$ with $\tau\in\{-\infty\}\cup\R$ be a time interval, $B :I \to \R^{d\times d}$ be a bounded continuous function of invertible matrices and $s:I\to [0,\infty)^d$ be a bounded continuous function.
  A (linear) \emph{nonautonomous compartmental system} with $d$ pools is given by an inhomogeneous linear nonautonomous differential equation
  \begin{equation}\label{eqn1}
    \dot x = B(t) x + s(t)\,,
  \end{equation}
  where we assume that the entries $\{b_{ij}(t)\}_{i,j\in\{1,\dots,d\}}$ of the matrix $B(t)$ satisfy
  \begin{enumer}
    \item[$\bullet$] $b_{ii}(t) < 0$ for all $i\in\{1,\dots,d\}$ and $t\in I$,
    \item[$\bullet$] $b_{ij}(t) \ge 0$ for all $i\not=j\in\{1,\dots,d\}$ and $t\in I$,
    \item[$\bullet$] $\sum_{i=1}^d b_{ij}(t) \le 0$ for all $j\in\{1,\dots,d\}$ and $t\in I$.
  \end{enumer}
\end{definition}

Let $\Phi:I\times I\to \R^{d\times d}$ denote the \emph{transition operator} of
the corresponding homogeneous equation $\dot x = B(t)x$, i.e.~the function $t\mapsto \Phi(t,t_0)x_0$ is the solution to $\dot x = B(t)x$ fulfilling the initial condition $x(t_0) = x_0$. Then the maximal solution to \eqref{eqn1} satisfying the initial condition $x(t_0) = x_0$ is given by
\begin{equation}\label{eqn5}
  \varphi(t,t_0,x_0) := \Phi(t,t_0)x_0 + \int_{t_0}^t \Phi(t,u)s(u)\,\rmd u \fa t\in I\,.
\end{equation}

In contrast to the autonomous case, nonautonomous compartmental systems of dimension two or higher are not explicitly solvable in general. Solutions can be obtained for systems with no feedbacks between pools, as the following example demonstrates.

\begin{example}[Explicitly solvable nonautonomous two-pool model]\label{exam3}
  The nonautonomous compartmental system
  \begin{displaymath}
    \dot x = \begin{pmatrix}
     b_{11}(t)& 0 \\
     b_{21}(t)  & b_{22}(t)
     \end{pmatrix}x
     + \begin{pmatrix}
         0 \\
         s_2(t)
       \end{pmatrix}\,,
  \end{displaymath}
  where $b_{11}(t), b_{22}(t)<0$, $b_{21}(t)\ge 0$ and $s_2(t)>0$ for all $t\in I$,
  can be solved explicitly as follows: the general solution of the first equation is given by
  \begin{displaymath}
     x_1(t)= x_1^0 \exp\big(\textstyle\int_{t_0}^t b_{11}(u) \,\rmd u\big)\,,
  \end{displaymath}
  and thus, the second equation reads as
  \begin{displaymath}
    \dot x_2 = b_{22}(t)x_2 + b_{21}(t)x_1^0 \exp\big(\textstyle\int_{t_0}^t b_{11}(u) \,\rmd u\big) + s_2(t)\,,
  \end{displaymath}
  and can be solved using \eqref{eqn5}, since the equation is one-dimensional.
\end{example}

\section{Exponential stability of nonautonomous compartmental systems}\label{sec2}

In this section, we provide a sufficient condition for global exponential stability of the nonautonomous compartmental system \eqref{eqn1}. This criterion will concern only the homogeneous part of \eqref{eqn1}, i.e.~the matrix-valued function $B$, from which stability for the inhomogeneous equation follows. Since the result holds also for linear systems which are not compartmental systems, we formulate it more generally.

\begin{theorem}[Sufficient condition for exponential stability]\label{theo1}
  Consider the linear nonautonomous differential equation
  \begin{equation}\label{eqn4}
    \dot x = B(t) x
  \end{equation}
  with transition operator $\Phi:I\times I\to \R^{d\times d}$.
  Suppose that the function $B$ is of the form
  \begin{equation}\label{blockform}
    B(t) = \begin{pmatrix}
     B_{11}(t)& 0 & 0 & &0\\
     B_{21}(t)& B_{22}(t) & 0 & &0\\
     B_{31}(t)& B_{32}(t) & B_{33}(t) & &0\\
     & & & \ddots & \\
     B_{m1}(t)& B_{m2}(t)& B_{m3}(t) & & B_{mm}(t)\\
    \end{pmatrix}
  \end{equation}
  for $m\ge 1$ with bounded functions $B_{ij}:I\to \R^{d_i\times d_j}$. Note that  $\sum_{i=1}^m d_i = d$.  We assume that the linear subsystems $\dot x_n = B_{nn}(t)x_n$, $n\in\{1,\dots,m\}$, are strictly diagonally dominant, i.e.~there exists a $\delta>0$ such that
  \begin{enumer}
    \item[(i)] $(B_{nn}(t))_{ii} < 0 $ for all $t\in I$ and $i\in\{1,\dots,d_n\}$,
    \item[(ii)] $(B_{nn}(t))_{ij} \ge 0 $ for all $t\in I$ and $i\not=j\in\{1,\dots,d_n\}$,
    \item[(iii)] $\sum_{j=1}^{d_n} (B_{nn}(t))_{ij} \le -\delta$ for all $t\in I$ and $i\in\{1,\dots,d_n\}$.
  \end{enumer}
  Then the linear system \eqref{eqn4} is exponentially stable, i.e.~there exist constants $K\ge1$ and $\gamma>0$ such that
  \begin{equation}\label{rel1}
    \|\Phi(t,t_0)\|\le K e^{-\gamma(t-t_0)} \fa t\ge t_0 > \tau \,.
  \end{equation}
\end{theorem}

\begin{proof}
  Assume first that $I$ is bounded below, and consider the linear systems
  \begin{equation}\label{eqn11}
    \dot x = B_{ii}(t)x
  \end{equation}
  for each $i\in\{1,\dots,m\}$. These systems are strictly diagonally dominant, and it follows from \cite[Proposition 3, page 55]{Coppel_78_1} that there exist $K_i\ge1$ and $\gamma_i>0$ such that
  \begin{displaymath}
     \|\Phi_i(t,t_0)\|\le K_i e^{-\gamma_i(t-t_0)} \fa t\ge t_0 > \tau \,,
  \end{displaymath}
  where $\Phi_i$ is the transition operator of \eqref{eqn11}. Next \cite[Theorem 4.1]{Poetzsche_16_1} or \cite[p.~540]{Battelli_15_1} yields that the dichotomy spectrum of \eqref{eqn4} is bounded above by $-\min_{i\in\{1,\dots,m\}} \gamma_i$. This in turn implies the claimed estimate \eqref{rel1}. In case $I=\R$, the results from \cite[Theorem 4.1]{Poetzsche_16_1} and \cite[p.~540]{Battelli_15_1} are not applicable directly, since they require the system to be defined on a half line, but the result follows by considering the two time intervals $(-\infty, 0)$ and $(0,\infty)$ separately (note that we are not interested in the dichotomy spectrum for the entire line, which is not determined by the block diagonal system; we only require an upper bound, which we get from the block diagonal system).
\end{proof}

The estimate \eqref{rel1} for the homogeneous system \eqref{eqn4} implies that any two solutions of the compartmental system \eqref{eqn1} converge to each other exponentially. More precisely, given two solutions $\mu_1, \mu_2:I\to\R^d$ of \eqref{eqn1}, then
\begin{displaymath}
  \|\mu_1(t)-\mu_2(t)\|\le K e^{-\gamma(t-t_0)} \|\mu_1(t_0)-\mu_2(t_0)\| \fa t\ge t_0 > \tau \,,
\end{displaymath}
which follows from the fact that the difference of these two solutions is a solution of the homogeneous system \eqref{eqn4}, for which the estimate \eqref{rel1} holds. This implies that any solution is \emph{forward attracting}, and in case the interval $I$ is unbounded below, then there also exists a unique \emph{pullback attracting} solution
\begin{equation}\label{eqn6}
  \nu(t):= \int^t_{-\infty} \Phi(t,u)s(u)\,\rmd u\fa t\in I\,,
\end{equation}
see \cite{Aulbach_96_1}. This solution pullback attracts bounded sets $B\subset \R^d$, in the sense of
\begin{displaymath}
  \lim_{t_0\to-\infty} \operatorname{dist}\big(\varphi(t,t_0, B), \{\nu(t)\}\big) = 0 \fa t\in I\,,
\end{displaymath}
where $\varphi$ denotes the maximal solution defined in \eqref{eqn5} and $\operatorname{dist}$ denotes the Hausdorff distance. We refer to \cite{Kloeden_11_2,Rasmussen_07_1} for an introduction to forward and pullback attractors of nonautonomous dynamical systems.

\section{The mean age system}\label{sec4}

We prove in this section that the mean ages of mass in a nonautonomous compartmental system are solutions of a linear nonautonomous differential equation, which we call the \emph{mean age system}. We derive this result from the evolution of age distributions, given by the well-known McKendrick--von F{\"o}rster equation \cite{McKendrick_26_1,Brauer_12_1,Thieme_03_1}, which is a linear first order partial differential equation. We also prove that the mean age system is exponentially stable under additional weak assumptions, by applying the theory developed in Section~\ref{sec2}.

The mean age system is pivotal for the analysis of transit times for nonautonomous compartmental systems, since in order to compute the average time the mass spends in the system, we do not need to look at the full age distribution of ages, but only at the mean ages.

Let $p_i(a,t)$ be the density function on age $a$ for the mass in pool $i$ at time $t$, where the age is the time since the mass entered the system. Note that the following formulation is valid in principle even if all rates are age-dependent, i.e.~$b_{ij}$ also depends on $a$, but we will not treat this situation here. The McKendrick--von F{\"o}rster  equation is given by
\begin{equation}
  \frac{\partial p_i}{\partial t}+\frac{\partial p_i}{\partial a}= \sum_{j=1}^d b_{ij}(t)p_j
\label{eq:1}
\end{equation}
with boundary condition
\begin{equation}\label{eq:2}
  p_i(0,t)=s_i(t)\,.
\end{equation}
Note that one also needs to specify initial conditions $p_i(a,0)$.

A componentwise solution can be written as follows.  Note that if there are no loops, the solution is explicit, otherwise it is only implicit.
The formula for $t>a$ is
\begin{align*}
  p_i(t,a) &= s_i(t-a)\exp(\textstyle\int_{t-a}^t b_{ii}(u)\,\rmd u)\\
  & \quad+ \sum_{j \neq i} \int_0^a\left(b_{ij}(t-\sigma)p_j(a-\sigma,t-\sigma)\exp(\textstyle\int_{t-\sigma}^tb_{ii}(u) \,\rmd u)\right)\,\rmd\sigma\,.
\end{align*}
Note that an analogous formula for $a<t$ exists.

We are particularly interested in the transit time of \eqref{eqn1} at a particular time $t$, which corresponds to the mean age of mass leaving the system at time $t$. For this purpose, we do not need the full age distribution determined by \eqref{eq:1}, since the situation is fully described by the mean age of mass in pool $i$, denoted as $\bar a_i (t)$. The following theorem says that the evolution of the mean ages is determined by an ordinary differential equation.

\begin{theorem}[Mean age system]
  Consider the nonautonomous compartmental system \eqref{eqn1} with a fixed solution $t\mapsto (x_1(t),\dots,x_d(t))$ of positive entries. Let $p_i(a,t)$ be the density function on age $a$ for the mass in pool $i$ at time $t$ (note that $\int_0^\infty p_i(a,t)\,\rmd a = x_i(t)$), and define the \emph{mean age} of mass in pool $i$ by
  \begin{displaymath}
    \bar a_i(t) = \frac{\int_0^\infty a p_i(a,t)\,\rmd a}{\int_0^\infty p_i(a,t)\,\rmd a} \fa i\in\{1,\dots,d\}\,.
  \end{displaymath}
  Then the mean ages $\bar a(t) =(\bar a_1(t), \dots, \bar a_d(t))$ solve the ordinary differential equation
  \begin{equation}\label{eqn2}
    \dot{\bar a} = g(t, x, \bar a)\,,
  \end{equation}
  with
  \begin{displaymath}
    g_i(t,x,\bar a) = 1+ \frac{\sum_{j=1}^d (\bar a_j-\bar a_i)b_{ij}(t) x_j(t)-\bar a_is_i(t)}{x_i(t)} \fa i\in\{1,\dots,d\}\,.
  \end{displaymath}
\end{theorem}

\begin{proof}
  By using $\int_0^\infty a \frac{\partial p_i}{\partial a}(t,a)\,\rmd a = -x_i(t)$ (integration by parts), it follows that
  \begin{align*}
    \dot{\bar a}_i(t) &= \frac{x_i(t) \int_0^\infty a \frac{\partial p_i}{\partial t}(t,a) \,\rmd a- \bar a_i (t)x_i(t)\dot x_i(t)}{x_i^2(t)}\\
    & = \frac{x_i(t) \int_0^\infty a \left(-\frac{\partial p_i}{\partial a}(t,a)+\sum_{j=1}^db_{ij}(t)p_j(t,a)\right)\,\rmd a - \bar a_i(t)x_i(t)\dot x_i(t)}{x_i^2(t)}\\
    & = \frac{x_i^2(t) + \sum_{j=1}^{d} b_{ij}(t) x_i(t) \int_0^\infty a p_j(t,a) \,\rmd a- \bar a_i(t)x_i(t)\dot x_i(t)}{x_i^2(t)}\\
    &=  1+ \frac{\sum_{j=1}^{d} b_{ij}(t) x_j(t)\bar a_j(t) - \bar a_i(t) \left(\sum_{j=1}^d b_{ij}(t)x_j(t) + s_i(t)\right)}{x_i(t)}\\
    &=  1+ \frac{\sum_{j=1}^d (\bar a_j(t)-\bar a_i(t))b_{ij}(t) x_j-\bar a_i (t)s_i(t)}{x_i(t)}\,.
  \end{align*}
  This finishes the proof.
\end{proof}

Combining the equations \eqref{eqn1} and \eqref{eqn2} yields
\begin{equation} \label{eqn3}
  \left(\begin{array}{c}
    \dot x \\ \dot{\bar a}
  \end{array}\right) = \left(\begin{array}{c}
    B(t)x + s(t)\\
   g(t, x, \bar a)
  \end{array}\right)\,,
\end{equation}
which is a $2d$-dimensional ordinary differential equation of skew product type, i.e.~the $x$-equation does not depend on $\bar a$, but the equation for $\bar a$ depends on $x$.
Note that \eqref{eqn3} is a nonlinear equation, but given a solution $x(t)= (x_1(t), \dots, x_d(t))$ of \eqref{eqn1}, the age equation \eqref{eqn2} is
an inhomogeneous linear nonautonomous differential equation, which reads as
\begin{displaymath}
  \dot{\bar a} = A(t, x(t)) \bar a + (1,\dots,1)^T\,,
\end{displaymath}
where
\begin{displaymath}\scriptstyle
  A(t, x(t))  = X(t)^{-1}\begin{pmatrix}\scriptstyle
   -s_1(t) - \sum_{j\not=1} b_{1j}(t)x_j(t) \hspace{-0.5cm}& \scriptstyle b_{12}(t)x_2(t) &  &\scriptstyle b_{1d}(t)x_d(t)\\
   \scriptstyle b_{21}(t)x_1(t) &\scriptstyle -s_2(t) - \sum_{j\not=2} b_{2j}(t)x_j(t) \hspace{-0.5cm} & &\scriptstyle b_{2d}(t)x_d(t)\\
   &  & \ddots & \\
   \scriptstyle b_{d1}(t)x_1(t) & \scriptstyle b_{d2}(t)x_2(t) &  & \hspace{-0.5cm}\scriptstyle -s_d(t) - \sum_{j\not=d} b_{dj}(t)x_j(t)
  \end{pmatrix}
\end{displaymath}
with $X(t):= \operatorname{diag} (x_1(t), \dots, x_d(t))$ for all $t\in I$.

We will show now that under additional weak assumptions, the mean age equation is exponentially stable.

\begin{theorem}[Exponential stability of the mean age system]\label{theo2}
  Consider the nonautonomous compartmental system \eqref{eqn1} with a fixed solution $t\mapsto x(t)=(x_1(t),\dots,x_d(t))$ of positive entries that are bounded and bounded away from zero, and suppose that \eqref{eqn1} satisfies the assumptions of Theorem~\ref{theo1} with $\delta>0$. In addition, assume that
  \begin{enumer}
    \item[(a)] $s_i(t) \ge \delta$ for all $t\in I$ and $i\in\{1,\dots,d_1\}$, and
    \item[(b)] for all $n\in\{2,\dots,m\}$ and $i\in\big\{1+\sum_{k=1}^{n-1}d_k,2+\sum_{k=1}^{n-1}d_k,\dots,\sum_{k=1}^{n}d_k\big\}$, there exists a $j\in\{1,\dots,\sum_{k=1}^{n-1}d_k\}$ such that $b_{ij}(t) \ge \delta$ for all $t\in I$.
  \end{enumer}
  Then the mean age system \eqref{eqn2} is exponentially stable. More precisely, there exist $\bar\delta \in(0,\delta)$ and $\bar K > 0$ such that the transition operator $\Psi: I\times I \to \R^{d\times d}$ of the homogenous equation $\dot {\bar a} = A(t, x(t)) \bar a$ satisfies the estimate
  \begin{displaymath}
    \|\Psi(t,t_0)\|\le \bar K e^{-\bar \delta (t-t_0)} \fa t\ge t_0 > \tau \,.
  \end{displaymath}
\end{theorem}

\begin{proof}
  We show now that the three conditions (i)--(iii) of Theorem~\ref{theo1} are satisfied with $\delta$ replaced by $\delta\min\{1, \min_{t\in I,i\in\{1,\dots,d\}} |x_i(t)|\}$. Note first that the matrix $A(t, x(t))$ has the same block decomposition as the matrix $B(t)$, which is described in \eqref{blockform}.

  Condition (i) of Theorem~\ref{theo1} follows from (a) (in case of $n=1$) or (b) (in case $n>1$; note that the sum of the entries in the $i$-th row of the matrix $A(t,x(t))$ equals to $-s_i(t)$, and (b) guarantees that the diagonal entry is negative even though $s_i(t)$ might be zero). Condition (ii) of Theorem~\ref{theo1} follows from the fact that the original system  \eqref{eqn1} is a compartmental system, and the solution $x(t)$ of \eqref{eqn1} has positive entries. Finally, condition (iii) of Theorem~\ref{theo1} follows from fact that the sum of the $i$-th row of the matrix $A(t,x(t))$ equals to $-s_i(t)$, and the positive contribution of at least $b_{ij}(t)x_j(t)\ge \delta \min_{t\in I,i\in\{1,\dots,d\}} |x_i(t)|$, with $i$ and $j$ chosen as in (b), will not be considered in the sum in condition (iii) of Theorem~\ref{theo1} and for this reason contributes negatively to this sum.
\end{proof}

A natural choice for the solution $t\mapsto x(t)$ in the above theorem is the exponentially stable solution defined in \eqref{eqn6} if the interval $I$ is unbounded below. If the interval $I$ is unbounded below, this will be the only bounded solution of the system, i.e.~the norm of all other solutions converges to $\infty$ in the limit $t\to-\infty$, so the solution \eqref{eqn6} is the only solution to which the theorem can be applied. However, if the interval $I$ is bounded below, then all solutions of the nonautonomous compartmental system \eqref{eqn1} are bounded and exponentially stable, and they are also bounded away from zero due to assumption (a) of Theorem~\ref{theo2}.

\section{Nonautonomous transit times}\label{sec5}

We define transit time as the mean age of mass leaving the system at a particular time $t$. Note that in our nonautonomous context, this quantity depends on the actual time $t$. We also provide a formula that corresponds to the mean age of mass currently residing in the compartmental system.

\begin{definition}[Nonautonomous transit time and mean age]\label{def1}
  Consider the skew product system \eqref{eqn3} consisting of the nonautonomous compartmental system \eqref{eqn1} and the mean age system \eqref{eqn2}. The \emph{transit time} of a solution $(x_1(t),\dots,x_d(t), \bar a_1(t), \dots, \bar a_d(t))$, $t\in I$, of this system is then defined as
  \begin{displaymath}
    R_t := \frac{\sum_{i=1}^d \bar a_i(t) x_i(t)\sum_{j=1}^d b_{ji}(t)}{\sum_{i=1}^d x_i(t)\sum_{j=1}^d b_{ji}(t)} \fa t\in I\,,
  \end{displaymath}
  and then \emph{mean age} of this solution is defined by
  \begin{displaymath}
    M_t := \frac{\sum_{i=1}^d \bar a_i(t) x_i(t)}{\sum_{i=1}^d x_i(t)} \fa t\in I\,.
  \end{displaymath}
\end{definition}

The transit time $R_t$ is the mean age of carbon leaving the system at time $t$, where as the mean age $M_t$ is the mean age of carbon in the system at time $t$.

Note that, in general, $R_t$ and $M_t$ are different, see Example~\ref{exam5} for the autonomous case. In the following example, we show that transit times and mean ages are the same for one-dimensional compartmental systems.

\begin{example}[Transit time and mean ages for one-dimensional compartmental systems]\label{exam1}
  Let $I\subset \R$ be an interval, and consider the one-dimensional nonautonomous compartmental system
  \begin{displaymath}
    \dot x = b(t) x +s(t)\,,
  \end{displaymath}
  where $b :I \to (-\infty,0)$ and $s:I\to(0,\infty)$ are bounded continuous functions. Fix a positive solution $t\mapsto x(t)$ of this system. Note that the solution is given explicitly by
  \begin{equation}\label{rel2}
     x(t)= x(t,t_0,x_0) = \exp\big(\textstyle\int_{t_0}^t b(u)\,\rmd u\big)x_0 +{\displaystyle \int_{t_0}^t }\exp\big(\textstyle\int_u^{t_0} b(v)\,\rmd v\big) s(u)\,\rmd u\,,
  \end{equation}
  where $t_0$ and $x_0$ are initial time and condition.
  Then the mean age equation is given by
  \begin{displaymath}
    \dot{\bar a} = -\frac{s(t)}{x(t)}\bar a + 1\,,
  \end{displaymath}
  and also this equation can be solved explicitly using \eqref{rel2}. Note that, in this one-dimensional context, the formulae for transit time and mean age from Definition~\ref{def1} are given by exactly the solution to this equation:
  \begin{displaymath}
    R_t= M_t = \bar a(t)\fa t\in I\,.
  \end{displaymath}
\end{example}

\section{Consistency with the autonomous case}\label{sec6}

In this section, we derive simple expressions for the transit time and mean age from Definition~\ref{def1} in the special case of an autonomous compartmental system. The expression for the autonomous transit time coincides with the heuristically obtained formula \eqref{restimeautonomous}, and we confirm the expression for the mean ages stated in \eqref{meanagesautonomous}.

Consider an autonomous compartmental system
\begin{equation}\label{ode1}
  \dot x = B x + s
\end{equation}
with an invertible matrix $B\in\R^{d\times d}$ and $s\in \R^d$. We assume that the homogeneous system $\dot x = Bx$ satisfies the assumptions of Theorems~\ref{theo1}~and~\ref{theo2}. Note that \eqref{ode1} has the exponentially stable equilibrium $x^* := -B^{-1}s$.

\begin{lemma}
  Consider the autonomous differential equation \eqref{ode1}. Then the mean age equation \eqref{eqn2} for the equilibrium $x^*$ reads as
  \begin{displaymath}
    \dot{\bar a} = \big(X^*\big)^{-1} B X^* \bar a + (1,\dots, 1)^T\,,
  \end{displaymath}
  and has the exponentially stable equilibrium
  \begin{displaymath}
    \bar a^* := -\big(X^*\big)^{-1}B^{-1}X^* (1, \dots, 1)^T\,,
  \end{displaymath}
  where $X^*:= \operatorname{diag} (x^*_1, \dots, x^*_d)$.
\end{lemma}

\begin{proof}
  Note that the mean age equation \eqref{eqn2} is given by
  \begin{displaymath}
    \dot{\bar a} =   \big(X^*\big)^{-1}\begin{pmatrix}\scriptstyle
     -s_1 - \sum_{j\not=1} b_{1j}x^*_j & \scriptstyle b_{12}x^*_2 &  &\scriptstyle b_{1d}x^*_d\\
     \scriptstyle b_{21}x^*_1 &\scriptstyle -s_2 - \sum_{j\not=2} b_{2j}x^*_j  & &\scriptstyle b_{2d}x^*_d\\
     &  & \ddots & \\
     \scriptstyle b_{d1}x^*_1 & \scriptstyle b_{d2}x^*_2 &  & \scriptstyle -s_d - \sum_{j\not=d} b_{dj}x^*_j
    \end{pmatrix}\bar a + \begin{pmatrix}
     1\\ \vdots\\ 1
    \end{pmatrix}\,.
  \end{displaymath}
  Since $Bx^* = -s$, we get $-s_i - \sum_{j\not= i} b_{ij}x^*_j = b_{ii} x^*_i$ for all $i\in\{1,\dots, d\}$, so the mean age equation \eqref{eqn2} gets simplified to
  \begin{align*}
    \dot{\bar a} &=   \big(X^*\big)^{-1}\begin{pmatrix}\scriptstyle
     b_{11}x^*_i & \scriptstyle b_{12}x^*_2 &  &\scriptstyle b_{1d}x^*_d\\
     \scriptstyle b_{21}x^*_1 &\scriptstyle b_{22}x^*_2  & &\scriptstyle b_{2d}x^*_d\\
     &  & \ddots & \\
     \scriptstyle b_{d1}x^*_1 & \scriptstyle b_{d2}x^*_2 &  & \scriptstyle b_{dd}x^*_d
    \end{pmatrix}\bar a + \begin{pmatrix}
     1\\ \vdots\\ 1
    \end{pmatrix}\\
    & = \big(X^*\big)^{-1} B X^* \bar a + (1,\dots, 1)^T\,.
  \end{align*}
  Hence the attractive equilibrium of \eqref{eqn2} is given by
  \begin{displaymath}
    \bar a^* := -\big(X^*\big)^{-1}B^{-1}X^* (1, \dots, 1)^T\,,
  \end{displaymath}
  which finishes the proof of this lemma.
\end{proof}

Let $\beta_i$ be the fraction of particles that enter the system from outside directly into pool $i$, i.e.
\begin{displaymath}
  \beta_i = \frac{s_i}{\sum_{i=1}^d s_i} \fa i\in\{1,\dots,d\}\,,
\end{displaymath}
and let $\beta=(\beta_1,\dots,\beta_d)^T$. Moreover, define $\eta = (\eta_1,\dots,\eta_d)^T$ by
\begin{displaymath}
  \eta_i = \frac{x_i^*}{\sum_{j=1}^d x_j^*} \fa i\in\{1,\dots,d\}\,,
\end{displaymath}
which describes how mass is distributed when the system is in equilibrium.
Note that $\sum_{i=1}^d \beta_i = \sum_{i=1}^d \eta_i= 1$.

\begin{proposition}[Autonomous transit times and mean ages]\label{theo3}
  Consider the autonomous compartmental system \eqref{ode1}. The transit time with respect to the equilibrium solution $t\mapsto(x^*, \bar a^*)$ does not depend on time and is given by
  \begin{displaymath}
    R= -(1,\dots,1) B^{-1} \beta\,,
  \end{displaymath}
  and the mean age of mass is given by
  \begin{displaymath}
    M= -(1,\dots,1) B^{-1} \eta\,.
  \end{displaymath}
\end{proposition}

\begin{proof}
  Using Definition~\ref{def1}, we have
  \begin{align*}
     R_t & = \frac{(1,\dots,1)B X^* \bar a ^*}{(1,\dots,1)B x^*}=- \frac{(1,\dots,1)X^* (1,\dots,1)^T}{\sum_{i=1}^d s_i}\\
     &= -\frac{(1,\dots,1)x^*}{\sum_{i=1}^d s_i} =-\frac{(1,\dots,1)B^{-1}s}{\sum_{i=1}^d s_i}\\
     & = -(1,\dots,1) B^{-1} (\beta_1,\dots,\beta_d)^T \fa t \in \R
  \end{align*}
  for the transit time and
  \begin{align*}
     M_t & = \frac{(1,\dots,1)X^* \bar a ^*}{(1,\dots,1) x^*}= -\frac{(1,\dots,1)B^{-1}X^*(1,\dots,1)^T}{\sum_{i=1}^d x_i^*}\\
     &= -\frac{(1,\dots,1)B^{-1}x^*}{\sum_{i=1}^d x_i^*}  = -(1,\dots,1) B^{-1} (\eta_1,\dots,\eta_d)^T \fa t \in \R
  \end{align*}
  for the mean age.  Note that both quantities do not depend on $t$, and this finished the proof of this proposition.
\end{proof}

Note that derivation of the autonomous quantities for transit time $R$ and mean age $M$ in Proposition~\ref{theo3} required the autonomous compartmental system \eqref{ode1} to be in equilibrium, and the classical approach to transit times, as outlined in Section~\ref{sec3}, is not applicable for autonomous systems not in equilibrium. It is very important to note that Definition~\ref{def1} is useful for autonomous systems also, since it is applicable to systems that are not in equilibrium. For such autonomous systems, transit times and mean ages will depend on time in general, and although they converge to $R$ and $M$ in the limit $t\to\infty$, they might be very different to $R$ and $M$.

\section{Mean ages and transit times for the CASA model}\label{sec7}

Here we illustrate predicted changes in the mean age of carbon leaving and remaining in the system for a terrestrial carbon model under a climate change scenario. We consider a modification of the CASA model as used in \cite{Buermann_07_1} globally without resolving the spatial details of carbon pools using nine pools representing the global terrestrial carbon (e.g.~three pools for plant biomass, or litter or soil organic matter). This caused the model to be precisely of the form of \eqref{eqn1}. Climate change was simulated by increasing atmospheric $\mathrm{CO_2}$ over time, which affected both $B(t)$ and $s(t)$ in \eqref{eqn1}. Increased $\mathrm{CO_2}$ directly increases carbon inputs $s(t)$ through carbon dioxide fertilization. They also directly increase mean global temperatures. This increases the carbon loss rates from some of the carbon pools, changing components of $B(t)$, and also has an effect on  $s(t)$. Thus increased $\mathrm{CO_2}$ alters the input and loss rates of components of the terrestrial carbon cycle, making both the sign and magnitude of the net change in carbon storage dependent the sensitivity of carbon inputs and loss rates.

We simulated changes in atmospheric $\mathrm{CO_2}$ using
\begin{displaymath} \label{eqnCO2}
  x_a(t) = 1715\exp\big(0.0305t/(1715+\exp(0.0305t)-1)\big)\,,
\end{displaymath}
where $x_a(t)$ is the atmospheric carbon dioxide concentration in parts per million and $t$ is years since the year $1850$. This represents a plausible time course of atmospheric $\mathrm{CO_2}$ from year 1850 ($t=0$) to 2500 ($t=650$) under a zero-mitigation, business as usual global change scenario \cite{Raupach_11_1} (illustrated in Fig.~\ref{fig1}a).

\begin{center}
  \begin{figure}[h]
    \includegraphics[width=12.5cm]{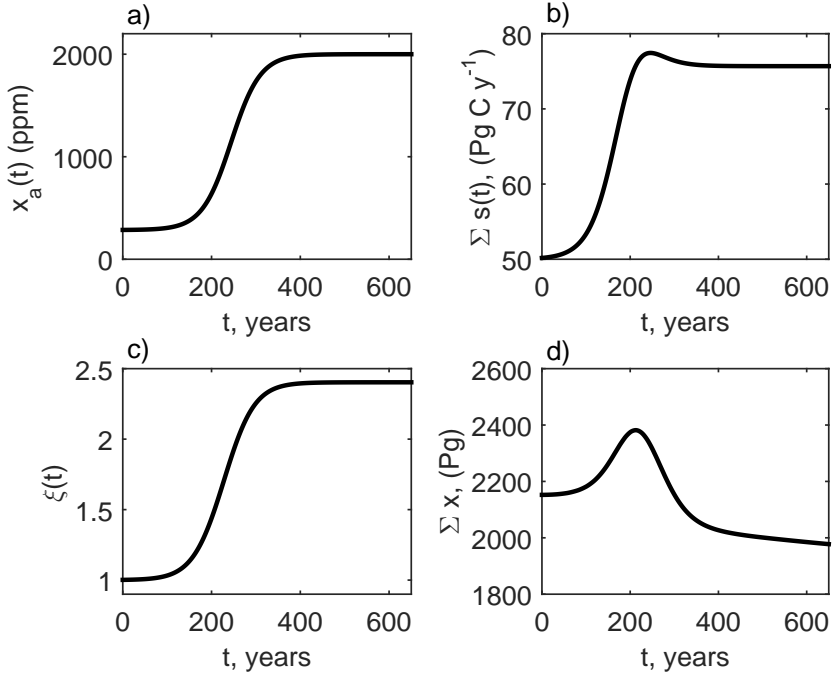}
        \caption{\label{fig1}Forcing functions and solution of the simplified CASA model. a) Nonautonomous dynamics are driven by changes in atmospheric $\mathrm{CO_2}$ over time as given by $x_a(t)$. b) The increased $\mathrm{CO_2}$ alters total carbon inputs per unit time via $\Sigma s(t)$. c) Increasing $\mathrm{CO_2}$ also increases temperatures which increases litter and soil carbon decomposition rates via $\xi(t)$. d) the resulting solution of total terrestrial carbon over time. Parameters for this model are as given in the text but also with $b_{11}=-0.{67}$, $b_{22}=-0.2$, $b_{33}=-0.04$, $b_{41}=0.5092$, $b_{42}=0.0260$, $b_{44}=-2.5$, $b_{51}=0.1608$, $b_{52}=0.1740$, $b_{55}=-0.4$, $b_{63}=0.04$, $b_{66}=-0.25$, $b_{74}=1.1250$, $b_{75}=0.1530$, $b_{76}=0.06$, $b_{77}=-0.7$, $b_{78}=0.0103$, $b_{79}=0.0002$, $b_{85}=0.042$, $b_{86}=0.07$, $b_{87}=0.3525$, $b_{88}=-0.023$, $b_{97}=0.0045$, $b_{98}=0.0001$, $b_{99}=-0.0004$.}
  \end{figure}
\end{center}

The effect of $\mathrm{CO_2}$  on mean global temperatures is modelled as
\begin{equation} \label{GlobalTemp}
  T_s(t)=T_{s0}+\frac{\sigma}{\ln(2)}\ln(x_a(t)/285)\,,
\end{equation}
where $T_{s0}=15$ is the mean land surface temperature in $1850$, and $\sigma$ is the sensitivity of global temperatures to $x_a(t)$. We chose an upper extreme of $\sigma=4.5$ based on the literature because the resulting simulation emphasises well the interplay between increased carbon input rates and carbon loss rates \cite{Scheffer_06_1}. Changes in carbon input rates are simulated using
\begin{equation} \label{eqnCarbonS}
  s(t) = (s_1(t), s_2(t), s_3(t), 0, 0, 0, 0, 0, 0)
\end{equation}
with $s_i(t)=f_i \alpha s_0(1+\beta(x_a(t), T_s(t))\ln(x_a(t)/285))$, where $f_i=0.33$ is the proportion of carbon input going to the different carbon pools,  $\alpha=0.5$ is the proportion of gross primary production that remains after respiration and $\beta$ is the sensitivity of $s(t)$ to $x_a(t)$ and $T_s(t)$, given by
\begin{displaymath} \label{CASABeta}
  \beta(x_a(t), T_s(t)) = \frac{3\rho x_a(t) \Gamma(T_s(t))}{(\rho x_a(t)-\Gamma(T_s(t)))(\rho x_a(t)+2\Gamma(T_s(t)))}\,,
\end{displaymath}
where $x=0.65$ is the ratio of the intracellular $\mathrm{CO_2}$ to $x_a(t)$, and $\Gamma(T_s(t))$ is given by
\begin{displaymath} \label{CASAGamma}
  \Gamma(T_s(t))=42.7+1.68(T_s(t)-25)+0.012(T_s(t)-25)^2\,,
\end{displaymath}
see \cite{Polglase_92_1}. The solution of \eqref{eqnCarbonS} with changes in $x_a(t)$ and $T_s(t)$ as described above is illustrated in Fig.~\ref{fig1}b.
The matrix controlling the rates of carbon transfer and loss from the system is given by
\begin{displaymath}\label{eqnCarbonB}{\small
  B(t) \!= \!\begin{pmatrix}
   b_{11}\hspace{-0.05cm}& 0 & 0 & 0 & 0 & 0 & 0 & 0 & 0 \\
   0  & b_{22}\hspace{-0.05cm} & 0 & 0 & 0 & 0 & 0 & 0 & 0 \\
   0  & 0 & b_{33}\hspace{-0.05cm} &  0 & 0 & 0 & 0 & 0 & 0 \\
   b_{41}\hspace{-0.05cm}  & b_{42}\hspace{-0.05cm} & 0 & b_{44}\xi(T_s(t))\hspace{-0.05cm} & 0 & 0 & 0 & 0 & 0 \\
   b_{51} \hspace{-0.05cm} & b_{52}\hspace{-0.05cm} & 0 & 0 & b_{55}\xi(T_s(t))\hspace{-0.05cm} & 0 & 0 & 0 & 0 \\
   0& 0& b_{63}\hspace{-0.05cm} & 0 & 0 & b_{66}\xi(T_s(t))\hspace{-0.05cm} & 0 & 0 & 0 \\
   0& 0& 0 & b_{74}\xi(T_s(t))\hspace{-0.05cm} & b_{75}\xi(T_s(t))\hspace{-0.05cm}  & b_{76}\xi(T_s(t))\hspace{-0.05cm} & b_{77}\xi(T_s(t))\hspace{-0.05cm} & b_{78}\xi(T_s(t))\hspace{-0.05cm}  & b_{79}\xi(T_s(t)) \\
   0& 0& 0 & 0 & b_{85}\xi(T_s(t))\hspace{-0.05cm}  & b_{86}\xi(T_s(t))\hspace{-0.05cm} & b_{87}\xi(T_s(t))\hspace{-0.05cm} & b_{88}\xi(T_s(t))\hspace{-0.05cm}  & b_{89}\xi(T_s(t)) \\
   0& 0& 0 & 0 & 0 & 0& b_{97}\xi(T_s(t))\hspace{-0.05cm} & b_{98}\xi(T_s(t))\hspace{-0.05cm}  & b_{99}\xi(T_s(t))
  \end{pmatrix}}\,,
\end{displaymath}
indicating that it is the loss rates of pools $i=\{3,\dots,9\}$ that change with time. The coefficients $b_{ij}$ are listed in the legend to Fig.~1, and
\begin{equation} \label{CASAQ10}
  \xi(T_s(t))=\xi_b^{0.1T_s(t)-2}\,,
\end{equation}
where $\xi(T_s(t))$ is the scaling of decomposition rates at $T_s=20$ degrees Celsius. Equation \eqref{CASAQ10} is illustrated in Fig.~\ref{fig1}c.

To define the model initial conditions we assume that $x_a(t)=x_a(0)$ for all $t<0$ and that $x(0)$ has reached the positive equilibrium solution of the resulting system of autonomous equations. The model is then simulated forward from this initial condition using \eqref{GlobalTemp} as the forcing function. Under this simulated scenario, total land carbon increases then decreases over time as shown in Fig.~1d. This would represent an initial net uptake of carbon from the atmosphere due to carbon dioxide fertilization followed ultimately by a net carbon loss from the land back to the atmosphere due to global warming (Fig.~\ref{fig2} shows how this carbon change over time is distributed amongst the different components of $x(t)$).

\begin{center}
  \begin{figure}[h]
    \includegraphics[width=12.5cm]{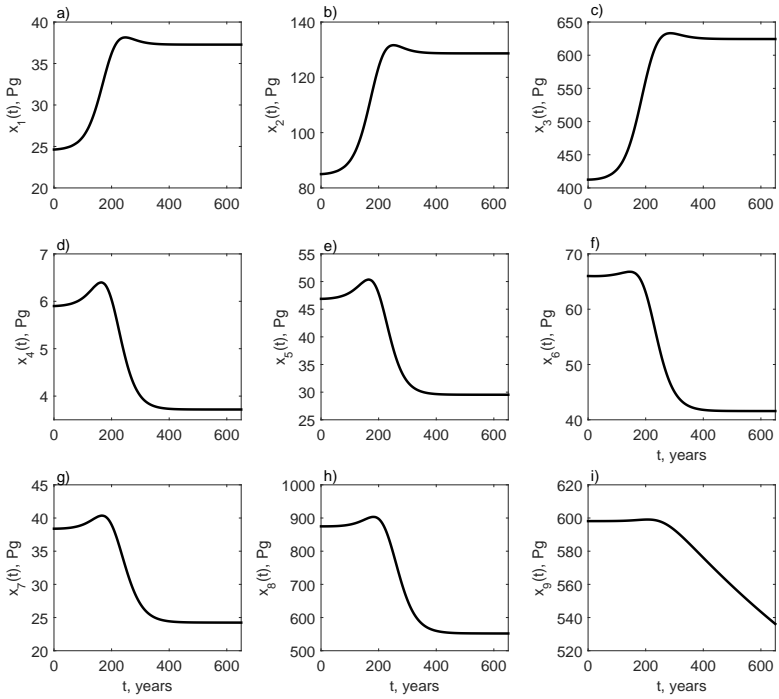}
    \caption{\label{fig2}Breakdown of the contributions of the different vegetation carbon pools to the change in the overall terrestrial carbon storage dynamics illustrated in Fig.~\ref{fig1}d. Pools $x_1$, $x_2$ and $x_3$ are carbon in leaves, roots and wood, respectively; pools $x_4$ to $x_6$ are carbon in different forms of litter, and pools $x_7$ to $x_9$ are carbon in different forms of soil.}
  \end{figure}
\end{center}

Calculations of the transit time $R_t$ and mean age $M_t$ of carbon in the system (according to Definition~\ref{def1}), for the nine-pool model for the climate change simulation described above, show an order of magnitude difference in the absolute values of $R_t$ and $M_t$ (Fig.~\ref{fig3}). This indicates that the average age of carbon stored on land is much older than the average age of carbon leaving the land. Note that at $t=0$ (corresponding to the year $1850$), we assumed that $M_0= M$, where $M$ is the mean age of the equilibrium solution at $t=0$ according to Proposition~\ref{theo3}.

\begin{center}
  \begin{figure}[h]
    \includegraphics[width=12.5cm]{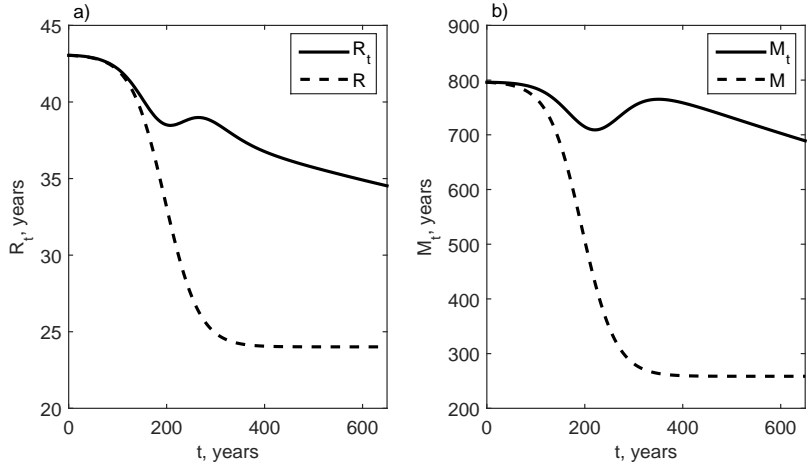}
    \caption{\label{fig3}Mean transit time $R_t$, and mean age $M_t$, compared with the instantaneous quantities $R$ and $M$.}
  \end{figure}
\end{center}

Perhaps surprisingly, the monotonic forcing of $B(t)$ and $s(t)$ translates into non-monotonic effects on $R_t$ and $M_t$. A detailed mathematical investigation of this phenomenon is outside the scope of the present study.

The nonautonomous properties $R_t$ and $M_t$ show contrasting trajectories to the instantaneous properties $R$ and $M$ (which we computed according to Proposition~\ref{theo3}, but note that, since the system is nonautonomous, the assumptions of this proposition are not fulfilled). For example the latter properties change monotonically over time. This must be because the long term outcome of an increase in the input rate of young carbon and an increase in the output rate of old carbon is a decrease in the age of carbon both leaving and remaining in the system. Over the course of the simulation the numerical values of the autonomous and nonautonomous properties become visibly different (Fig. 3). This is because it will take a long time for the values of $R_t$ and $M_t$  to approach $R$ and $M$ due to the small loss rate of the ninth soil pool.

\section{Conclusions}

Models for terrestrial carbon cycling have led to renewed interest in the properties of compartment models. Key quantities that have been studied over many years in compartment models with parameters fixed in time \cite{Eriksson_71_1,Bolin_73_1,Anderson_83_1} are the mean age of particles in the system and the transit time of particles leaving the system.  Formulae for these quantities that give the mean age and transit time in terms of parameters of the system in the long time limit have led to insights, but cannot be applied to the case of changing parameters.

As parameters change, for example in a model of carbon cycling due to climate change, it is not correct to calculate the mean age or transit time from the instantaneous parameter values.  Using the theory of nonautonomous differential equations as a tool, and beginning with time dependent age structured models, we are able to define and derive formulae for the transit time and mean age for particles in the case of temporally changing parameters.  These definitions lead to quantities that reduce to the analogous formulae for the autonomous (constant parameter) case when parameters do not change in time.  However, the formulae for the nonautonomous case also highlight the fact that even in the constant parameter case the transit time and mean age do depend on initial conditions; some of the standard formulae do not include this dependence.

The difference between a transit time or mean age that is computed based on the parameters at a given instant and the better approach of taking into account the history of the system can be substantial as we illustrate using a variant of the CASA model. Thus, the approach we develop here is not just of mathematical interest but is of substantial practical importance as well.

\bigskip

\begin{acknowledgements}
  Martin Rasmussen was supported by an EPSRC Career Acceleration Fellowship EP/I004165/1 (2010--2015) and by funding from the European Union's Horizon 2020 research and innovation programme for the
  ITN CRITICS under Grant Agreement number 643073. Alan Hastings was supported by Army Research Office grant  W911NF-13-1-0305. Katherine E.O.~Todd-Brown is grateful for the support of the Linus Pauling Distinguished Postdoctoral Fellowship program which is funded under the Laboratory Directed Research and Development Program at Pacific Northwest National Laboratory, a multiprogram national laboratory operated by Battelle for the U.S. Department of Energy. Ying Wang was supported by a Ralph E.~Powe Junior Faculty Enhancement Award from Oak Ridge Associated Universities and by a Faculty Investment Program and a Junior Faculty Fellow Program grant from the Research Council and College of Arts and Sciences of the University of Oklahoma Norman Campus. This work was assisted through participation of the authors in the working group \emph{Nonautonomous Systems and Terrestrial Carbon Cycle}, at the \emph{National Institute for Mathematical and Biological Synthesis}, an institute sponsored by the National Science Foundation, the US Department of Homeland Security, and the US Department of Agriculture through NSF award no.~EF-0832858, with additional support from The University of Tennessee, Knoxville.
\end{acknowledgements}

\providecommand{\bysame}{\leavevmode\hbox to3em{\hrulefill}\thinspace}
\providecommand{\MR}{\relax\ifhmode\unskip\space\fi MR }
\providecommand{\MRhref}[2]{%
  \href{http://www.ams.org/mathscinet-getitem?mr=#1}{#2}
}
\providecommand{\href}[2]{#2}

\end{document}